\documentclass{amsart}
\usepackage{amssymb,amsmath, amsfonts, amsrefs, tikz, epsfig, float}
\usepackage{amsthm}
\usepackage{mathrsfs}
\usepackage{enumitem, indentfirst}
\usepackage[fleqn,tbtags]{mathtools}
\usepackage{color}
 \newtheorem{theorem}{Theorem}[section]
 \newtheorem{corollary}[theorem]{Corollary}
 \newtheorem{lemma}[theorem]{Lemma}
 \newtheorem{proposition}[theorem]{Proposition}

 \theoremstyle{definition}

 \theoremstyle{remark}
 
  \numberwithin{equation}{section}

\renewcommand{\epsilon}{\varepsilon}
\renewcommand{\phi}{\varphi}
\renewcommand{\theta}{\vartheta}
\DeclareMathOperator{\sform}{\mathfrak{s}}
\DeclareMathOperator{\tform}{\mathfrak{t}}

\DeclareMathOperator{\wform}{\mathfrak{w}}

\DeclarePairedDelimiterX\sipt[2]{(}{)_{\tform}}{#1\,\delimsize\vert\,#2}
\DeclarePairedDelimiterX\sipv[2]{(}{)_{v}}{#1\,\delimsize\vert\,#2}
\DeclarePairedDelimiterX\sipw[2]{(}{)_{w}}{#1\,\delimsize\vert\,#2}

\newcommand{\alg}{\mathscr{A}}

\newcommand{\dupN}{\mathbb{N}}

\newcommand{\seq}[1]{(#1_{n})_{n\in\dupN}}

\newcommand{\nen}{n\in\mathbb{N}}
\newcommand{\dupR}{\mathbb{R}}
\newcommand{\dupC}{\mathbb{C}}

\newcommand{\ran}{\operatorname{ran}}

\newcommand{\lef}{\mathscr{L}(E,F)}
\newcommand{\lefpoz}{\mathscr{L}_+(E,F)}
\newcommand{\lhaf}{\mathscr{L}(\hila,F)}

\newcommand{\leha}{\mathscr{L}(E,\hila)}

\newcommand{\hil}{\mathcal H}

\newcommand{\hila}{\hil_A}
\newcommand{\hilc}{\hil_C}

\DeclarePairedDelimiterX\sip[2]{(}{)}{#1\,\delimsize\vert\,#2}
\DeclarePairedDelimiterX\siptilde[2]{(}{)_{\!_{\widetilde{A}}}}{#1\,\delimsize\vert\,#2}
\DeclarePairedDelimiterX\sipf[2]{(}{)_{f}}{#1\,\delimsize\vert\,#2}
\DeclarePairedDelimiterX\sipg[2]{(}{)_{g}}{#1\,\delimsize\vert\,#2}
\DeclarePairedDelimiterX\siptw[2]{(}{)_{\tform+\wform}}{#1\,\delimsize\vert\,#2}
\DeclarePairedDelimiterX\set[2]{\{}{\}}{#1\,:\,#2}
\DeclarePairedDelimiterX\dual[2]{\langle}{\rangle}{#1,#2}
\DeclarePairedDelimiterX\sipa[2]{(}{)_{\!_A}}{#1\,\delimsize\vert\,#2}
\DeclarePairedDelimiterX\sipc[2]{(}{)_{\!_C}}{#1\,\delimsize\vert\,#2}
\DeclarePairedDelimiterX\sipab[2]{(}{)_{\!_{A+B}}}{#1\,\delimsize\vert\,#2}
\DeclarePairedDelimiterX\sipb[2]{(}{)_{\!_B}}{#1\,\delimsize\vert\,#2}
\newcommand{\anti}[1]{\bar{#1}'}

\newcommand{\limn}{\lim\limits_{n\rightarrow\infty}}

\allowdisplaybreaks
\title[Lebesgue decomposition via Arlinskii's iteration]{Operators on anti-dual pairs:\\ Lebesgue decomposition via Arlinskii's iteration}

\author[\'A. G\"ode]{\'Abel G\"ode}

\address{%
\'A. G\"ode \\ Department of Applied Analysis  and Computational Mathematics\\ E\"otv\"os Lor\'and University\\ P\'azm\'any P\'eter s\'et\'any 1/c.\\ Budapest H-1117\\ Hungary}
\email{godeabel@student.elte.hu}

\author[Zs. Tarcsay]{Zsigmond Tarcsay}

\address{%
Zs. Tarcsay \\ Department of Mathematics\\ Corvinus University of Budapest\\ IX. F\H ov\'am t\'er 13-15.\\ Budapest
H-1093 \\ Hungary\\ and Department of Applied Analysis  and Computational Mathematics\\ E\"otv\"os Lor\'and University\\ P\'azm\'any P\'eter s\'et\'any 1/c.\\ Budapest H-1117\\ Hungary}
\email{tarcsay.zsigmond@uni-corvinus.hu}

\thanks{ 
Project no. TKP 2021-NVA-09 has been implemented with the support provided by the Ministry of Innovation and Technology of Hungary from the National Research, Development and Innovation Fund, financed under the TKP2021-NVA funding scheme.}

\subjclass[2020]{Primary 47B65, 47A07}

\keywords{anti-dual pair; positive operator; parallel sum; Lebesgue decomposition; singularity}

\begin{document}

\begin{abstract}  

The aim of this paper is to prove a general Lebesgue decomposition theorem for positive operators on so-called anti-dual pairs, following the iterative approach introduced by Arlinskii. This procedure and the resulting theorem encompass several special cases, including positive operators on Hilbert spaces, non-negative forms on vector spaces, and representable functionals over *-algebras.

\end{abstract}


\maketitle

\section{Introduction}

The theory of non-commutative Lebesgue decomposition dates back at least to the pioneering paper by T. Ando \cite{ando1976lebesgue}. According to his results, any positive operator on a Hilbert space can be decomposed into the sum of absolutely continuous and singular parts with respect to another operator. In Ando's method, the absolutely continuous part in the decomposition is constructive and can be determined using the parallel sum.

In contrast, Arlinskii \cite{arlinskii2017} used an iterative process to produce the singular part. However, in his approach, the key concept was also the parallel sum of operators.

The parallel sum first appeared in the work of Anderson and Duffin \cite{anderson1969series} in the context of matrices (see also \cite{eriksson1986potential}). Since then, analogous operations have been introduced in various other structures, such as bounded positive operators on Hilbert spaces \cite{pekarev1976parallel}, non-negative forms \cite{hassi2009lebesgue}, finitely additive measures \cite{titkos2013content}, states over *-algebras \cite{tarcsay2015parallel}, and positive operators over so-called anti-dual pairs \cite{TARCSAY2020Lebesgue}. One common essential property of the parallel sum operation across all these structures is that it can be used to characterize singularity. This feature was exploited by Arlinskii \cite{arlinskii2017} to prove Ando's Lebesgue decomposition theorem, finding the singular part as the fixed point of a single-variable map. Titkos \cite{titkos2019arlinskii} imitated this iterative procedure for non-negative sesquilinear forms and provided a simple and elementary proof of the existence of the Lebesgue decomposition.

As Ando noted, the Lebesgue decomposition in non-commutative structures is generally not unique. However, the decomposition Ando constructed is distinctive because the corresponding absolutely continuous part is maximal in a certain sense. The decomposition obtained via Arlinskii's iterative method coincides with this distinguished one. In Titkos's paper \cite{titkos2019arlinskii}, this extremality of the obtained Lebesgue decomposition is not demonstrated, but in many cases (when the decomposition is unique), his result is still effectively applicable.

The aim of the present paper is to prove a Lebesgue decomposition theorem over a more general structure (generalizing both bounded operators on Hilbert spaces and sesquilinear forms) based on the aforementioned iterative procedure, for positive operators on so-called \textit{anti-dual pairs}. Such a decomposition was already established by the second-named author in \cite{TARCSAY2020Lebesgue}; however, the proof presented here is considerably shorter. Moreover, we also establish the maximality of the decomposition in the sense of Ando, making our result a joint generalization of the results of Arlinskii and Titkos. In addition, we apply our main result to derive a corresponding Lebesgue decomposition theorem for positive functionals on a *-algebra.

\section{Preliminaries}

Throughout the paper, let $E$ and $F$ denote complex vector spaces which are intertwined via a separating, sesquilinear function
\begin{equation*}
    \langle\cdot,\cdot\rangle:F\times E\to\mathbb{C},
\end{equation*}
called anti-duality function. We shall refer to the triple $(E,F,\langle\cdot,\cdot\rangle)$ (shortly denoted by $\langle F,E\rangle$) as an anti-dual pair. 

One of the most natural examples for anti-dual pairs are Hilbert spaces.  Indeed, if $\hil$ is a Hilbert space, then under the choice $E=F=\hil$ the pair $\sip \hil\hil$ forms an anti-dual pair, where of course the anti-duality function is the inner product  $\sip\cdot\cdot$ of $\hil$. 


A more general example for an anti-dual pair is the triple $(X,\anti{X},\langle\cdot,\cdot\rangle)$, where $X$ denotes a locally convex Hausdorff space, $\anti{X}$ is its conjugate topologic dual space (that is, the set of all continuous and anti-linear functionals $f$ from $X$ to $\dupC$) and the anti-duality function is the evaluation given by
\begin{equation*}
    \langle f,x\rangle\coloneqq f(x),\quad x\in X,f\in\anti{X}.
\end{equation*}

Let now $\dual FE$ be an anti-dual pair,  then the weak topology $w\coloneqq w(E,F)$ on $E$ induced by the family $\set{\dual f\cdot }{f\in F}$ is  a locally convex Hausdorff topology satisfying the  canonical duality condition
\begin{equation*}
    \anti{E}=F.
\end{equation*}
In other words, the topological anti-dual of $E$ can be identified with $F$ along the identification $f\equiv \dual f\cdot $.   
Similarly, the weak topology $w^*\coloneqq w(F,E)$ on $F$ induced by $E$ is also a locally convex Hausdorff topology such that $F'=E$. 
These duality identifications imply in particular  that the canonical anti-dual pairs presented in the previous example are the most general examples of anti-dual pairs. 

Now we can introduce the adjoint of a weakly continuous operator. Let $\langle F_1,E_1\rangle$ and $\langle F_2,E_2\rangle$ be anti-dual pairs and equip each member of the spaces with the corresponding weak topology.  The adjoint of a weakly continuous linear operator $T:E_1\to F_2$ is a  linear operator $T^*:E_2\to F_1$ that satisfies
\begin{equation*}
    \langle Tx^{}_1,x^{}_2\rangle^{}_2=\overline{\langle T^*x^{}_2,x^{}_1\rangle^{}_1},\quad x^{}_1\in E^{}_1,x^{}_2\in E^{}_2.
\end{equation*}
It is immediate that $T^*$ itself is weakly continuous between $E_2$ and $F_1$ such that $(T^*)^*=T$. 

In what follows, the set of weakly continuous linear operators $T$ between $E$ and $F$ will be denoted by $\lef$.
Clearly, if $T\in \lef$, then also $T^*\in \lef$. This allows us to introduce the notion of \textit{self-adjointness} of a weakly continuous linear operator $T\in\lef$ in a self-evident way. Note that  $T^*=T$  if and only if
\begin{equation}\label{E:self-adj}
    \dual{Tx}{y}=\overline{\dual{Ty}{x}},\qquad x,y\in E.
\end{equation}
We also worth to note that any linear operator $T:E\to F$ satisfying \eqref{E:self-adj} is automatically weakly continuous and self-adjoint. Furthermore, it is also readily seen that $T\in\lef$ is self-adjoint if and only if its \textit{quadratic form} is real, that is, $\dual{Tx}{x}\in \dupR$ for every $x\in E$.

An important subclass of weakly continuous operators, which is the main focus of this paper, is the cone of positive operators. A linear operator $A:E\to F$ is called positive if it satisfies
\begin{equation*}
    \dual{Ax}{x}\geq 0,\qquad x\in E.
\end{equation*}
Note that every positive operator is weakly continuous and self-adjoint (see e.g. \cite{TarcsayMathNach}*{Proposition 2.3}). The cone of positive operators will be denoted by the symbol $\lefpoz$.  

Considering a weakly continuous linear operator $T:E\to\mathcal{H}$ mapping $E$ to a Hilbert space $\mathcal{H}$, its adjoint acts as $T^*:\mathcal{H}\to F$, moreover, their composition, $T^*T:E\to F$ satisfies
\begin{equation}
    \dual{T^*Tx}{x}=\sip{Tx}{Tx},\quad x\in E.
\end{equation}
It is immediate that  $T^*T\in\lefpoz$ is a positive operator.

In the following we shall show that every positive operator $A\in\lefpoz$ admits a factorization $A=T^*T$ through an associated Hilbert space. However, this can only be done under an extra topological condition on $\dual FE$. We say that the anti-dual pair $\dual FE$ is $w^*$-sequentially complete if the topological vector space $(F,w(F,E))$ is sequentially complete. We remark that $w^*$-sequentially completeness is a kind of `weak Banach-Steinhaus property' for the anti-dual pair: it says that if the sequence $\seq f$ of $F$ is pointwise Cauchy (that is, $\dual{f_n}{x}$ is a Cauchy sequence for every $x$ in $E$), then there is an $f\in F$ such that $\dual{f_n}{x}\to \dual{f}{x}$ for every $x\in E$. As a corollary one concludes that if $X$ is a Banach space (or, more generally, if $X$ is a barrelled space), then $(X,\anti X)$ forms a $w^*$-sequentially complete  anti-dual pair. Another, much more obvious example for a $w^*$-sequentially complete anti-dual pair is $\dual{\bar{X}^*}{X}$, where $X$  is an arbitrary vector space and $\bar{X}^*$ is its algebraic anti-dual.
 
  Let now $\dual{F}{E}$ be a  $w^*$-sequentially complete anti-dual pair and let $A\in\lefpoz$.  The range  $\ran A$ can be equipped with an natural inner product  
\begin{equation*}
    \sipa{Ax}{Ay}:=\dual{Ax}{y},\quad x,y\in E
\end{equation*}
so that $(\ran A,\sipa\cdot\cdot )$ becomes a pre-Hilbert space. 
Let $\hila$ denote its completion and consider the canonical embedding operator 
\begin{equation}
    J_A(Ax):=Ax,\quad x\in E
\end{equation}
of $\ran A\subseteq \hila$ into $F$. It can be shown that $J_A$ is weakly continuous, so that, by $w^*$-sequentially completeness of $\dual FE$, it uniquely extends to a continuous weakly continuous operator $J_A\in\lhaf$. Its adjoint  $J_A^*\in\leha$ satisfies
\begin{equation}\label{E:JA*x}
    J_A^*x=Ax,\quad x\in E,
\end{equation}
that implies the following useful factorization of $A$:
\begin{equation}\label{E:factorization}
    A=J_A^{}J_A^*.
\end{equation}
Details of the above construction can be found in \cite{TARCSAY2020Lebesgue}. 


To conclude this section we introduce some important concepts and results from the theory of Lebesgue-type decomposition of operators. Given two positive operators $A,B\in\lefpoz$  we say that $B$ is absolutely continuous with respect to $A$ (notated by $B\ll A$), if for any sequence $(x_n)_{n\in\mathbb{N}}\subset E$, $\dual{Ax_n}{x_n}\to 0$ and $\dual{B(x_n-x_m)}{x_n-x_m}\to 0$ imply $\dual{Bx_n}{x_n}\to 0$. $A$ and $B$ are called mutually singular (notated by $A\perp B$), if for any positive operator $C\in\lefpoz$, the relations $C\le A$ and $C\le B$ imply $C=0$. It is important to note that $B\ll A$ if and only if there exist an increasing sequence of positive operators $(B_n)_{n\in\mathbb{N}}$ strongly converging to $B$ and a sequence $(\alpha_n)_{n\in\mathbb{N}}$ of nonnegative numbers satisfying $B_n\le\alpha_nA$ for all $n\in\mathbb{N}$ (see \cite{TARCSAY2020Lebesgue}*{Theorem 5.1}).

In \cite{TARCSAY2020Lebesgue}*{Theorem 3.3} it was proved that there always exists a decomposition of $B$ to the sum of operators $B_a$ and $B_s$ such that  $B_a\ll A$ and $B_s\perp A$. According to \cite{TARCSAY2020Lebesgue}*{Theorem 7.2}, such a decomposition is not unique in general, but there exists a distinguished decomposition $B=B_a+B_s$ to absolutely continuous and singular parts, where the absolutely continuous part $B_a$ has the following maximal property:
\begin{equation}\label{E:Ba}
    B_a=\max\set{C\in\lefpoz}{C\le B,C\ll A}.
\end{equation}
Following Ando's notations in \cite{ando1976lebesgue}, we will  denote the maximal absolute continuous operator in \eqref{E:Ba} by $[A]B:=B_a$, and we shall call it the $A$-absolutely continuous part of $B$.

Another important notion related to the Lebesgue decomposition theory is the so-called parallel sum. In \cite{TARCSAY2020Lebesgue}*{Theorem 4.1} it was proved that for any $A,B\in\lefpoz$ positive operators there exists a unique positive operator $A:B$ satisfying
\begin{equation}
    \dual{(A:B)x}{x}=\inf\set{\dual{A(x-y)}{x-y}+\dual{By}{y}}{y\in E},\quad x\in E.
\end{equation}
Moreover, the parallel sum is a commutative operation satisfying $A:B\le A$ and $A:B\le B$. It is also monotonous in the sense that $A_1\le A_2$ and $B_1\le B_2$ imply $A_1:B_1\le A_2:B_2$. The significance of the parallel sum in Lebesgue decomposition theory is that it can be used to obtain the absolute continuous part, namely:
\begin{align*}
    \dual{[A]Bx}{y}=\limn \dual{(A:nB)x}{y},\qquad x,y\in E.
\end{align*}
The singularity can also be characterized in terms of the parallel sum:  $A$ and $B$ are singular if and only if $A:B=0$ (see \cite{TARCSAY2020Lebesgue}*{Theorem 6.1}).





\section{Arlinskii's iteration}

In this section, we will apply an iterative procedure, based on the approach of Arlinskii \cite{arlinskii2017}, to produce a Lebesgue decomposition of a positive operator with respect to another. A similar procedure was used by Titkos \cite{titkos2019arlinskii} to construct the Lebesgue decomposition of non-negative sesquilinear forms. The context we consider here can be seen as a common generalization of both methods, as we will demonstrate in the next section.

The key feature of this method is that it identifies the singular part of the Lebesgue decomposition as the fixed point of an appropriately defined operation. The fixed point itself (i.e., the singular part) is then obtained as the strong limit of the operator sequence generated by iterating this operation.

For a fixed positive operator $A\in\lefpoz$ let us consider the map $\mu_A:\lefpoz\to\lefpoz$ defined as follows
\begin{equation*}
    \mu_A(X)\coloneqq X-X:A.
\end{equation*}
A positive operator $X$ is $A$-singular if and only if $\mu_A(X)=X$, because $X\perp A$ is equivalent to $X:A=0$. 

Let us consider now another positive operator $B$ and introduce the sequence $\seq B$ by letting 
\begin{equation}\label{E:iteration}
    B_0\coloneqq \mu_A^{[0]}(B)\coloneqq B,\qquad B_{n+1}\coloneqq \mu_A^{[n+1]}(B)\coloneqq \mu^{}_A(\mu_A^{[n]}(B))=B_n-B_n:A.
\end{equation}
The main purpose of this paragraph is to show that the sequence $\seq B$ converges pointwise to a positive operator $B_s$, that is the identical with  the fixed point of the operation $\mu_A$. Furthermore, one has 
\begin{equation*}
    B_s=B-[A]B.
\end{equation*}
 In other words, the decomposition $B=(B-B_s)+B_s$ is the is identical with the Lebesgue decomposition proved in the \cite{TARCSAY2020Lebesgue}*{Theorem 3.3}.
 
To begin with we prove a weaker version of the above statement. However, this statement gives a short and simple proof of the existence of the Lebesgue decomposition in the context of anti-dual pairs.
\begin{proposition}\label{P:3.1}
    There exists a strong limit of the sequence $(B_n)_{n\in\mathbb{N}}$ denoted by $B_s$, moreover $B_s$ is $A$-singular and $B_a\coloneqq B-B_s$ is $A$-absolutely continuous.
\end{proposition}
\begin{proof}
    Since $0\le X:G\le X$ for any $X,G$ positive operators we have $B\ge B_n\ge B_{n+1}\ge 0$ for any $n\in\mathbb{N}$. The sequence $\langle B_nx,x\rangle$ consists of nonnegative real numbers and monotonely decreases, therefore it is convergent. Using the well-known polarization formula for sesquilinear functions we have
    \begin{equation*}
        \langle B_mx-B_nx,y\rangle=\frac{1}{4}\sum_{k=0}^3i^k\langle(B_m-B_n)(x+i^ky),x+i^ky\rangle
    \end{equation*}
    which implies that for any $x,y\in E$, the sequence $\langle B_nx,y\rangle$ converges too. Since the anti-dual pair $\langle F,E\rangle$ is $w^*$-sequentially complete, it follows, that the functionals
    \begin{equation*}
        f_x:E\to\mathbb{C};\quad y\mapsto\lim_{n\to\infty}\langle B_nx,y\rangle
    \end{equation*}
    are continuous for every $x\in E$, so they belong to $F$ accordingly. Consider the operator
    \begin{equation*}
        B_s:E\to F;\quad x\mapsto f_x.
    \end{equation*}
    It is immediate that $B_s$ is a positive (hence continuous)  operator.

    First we show that 
    $B_s$ is $A$-singular. Indeed,  we have
    \begin{equation*}
        A:B_s=\lim_{n\to\infty}(A:B_n)=\lim_{n\to\infty}(B_n-B_{n+1})=0,
    \end{equation*}
    which implies that $A$ and $B_s$ are mutually singular due to \cite{TARCSAY2020Lebesgue}*{Theorem 6.1}.
    
    To see that $B_a\coloneqq B-B_s$ is absolutely continuous with respect to $A$, it suffices to notice that $B_a$ is the strong limit of the   monotone increasing sequence $(B-B_n)_{n\in\mathbb{N}}$, where every  member of this sequence can be estimated as follows: 
    \begin{equation*}
        B-B_n=\sum_{k=0}^{n-1}B_k:A\le nA.
    \end{equation*}
    Hence $B_a\ll A$ due to \cite{TARCSAY2020Lebesgue}*{Theorem 5.1}. 
\end{proof}

Consider now the positive operators $A,B\in\lefpoz$ and let $C\coloneqq A+B$. Since have $A,B\leq C$, it follows that there exist two positive contractions $\widetilde{A},\widetilde{B}$ on the auxiliary Hilbert space $\hilc$ such that 
\begin{equation*}
    A=J_C^{}\widetilde{A}J_C^* \qquad\mbox{and}\qquad B=J_C^{}\widetilde{B}J_C^*
\end{equation*}
Note that $\widetilde{A}+\widetilde{B}=I_C$ where $I_C$ stands for the identity operator on $\mathcal{H}_C$.

 The following lemma shows the relationship between the parallel sum of the operators $A,B\in \lefpoz$ and of the associated contractions $\widetilde{A}, \widetilde{B}\in\mathscr{B}_+(\hilc)$:
\begin{lemma}\label{L:2.1}
    Let $C$ denote $A+B$ and $J_C$ be the embedding of Hilbert space $\mathcal{H}_C$ into $F$, then we have
    \begin{equation*}
        J^{}_C(\widetilde{A}:\widetilde{B})J_C^*=A:B.
    \end{equation*}
\end{lemma}
\begin{proof}
    Using the density of $\ran J_C^*$ in $\mathcal{H}_C$ we obtain
    \begin{align*}
        \dual{ J_C(\widetilde{A}:\widetilde{B})J_C^*x}{x}&=\sipc{\widetilde{A}:\widetilde{B}(J_C^*x)}{J_C^*x}\\&
        =\inf_{\xi\in\mathcal{H}_C}\{\sipc{\widetilde{A}(J_C^*x+\xi}{J_C^*x+\xi}+\sipc{\widetilde{B}\xi}{\xi}\}\\
        &=\inf_{y\in E}\{\sipc{\widetilde{A}(J_C^*x+J_C^*y)}{J_C^*x+J_C^*y}+\sipc{\widetilde{B}J_C^*y}{J_C^*y}\}\\
        &=\inf_{y\in E}\{\dual{J_C\widetilde{A}(J_C^*x+J_C^*y)}{x+y}+\dual{ J_C^{}\widetilde{B}J_C^*y}{y}\}\\
        &=\dual{(A:B)x}{x} 
    \end{align*}
    for all $x\in E$, hence the statement of the lemma is proved.
\end{proof}

Let us introduce now the operator sequence $(\widetilde{B}_n)_{\nen}$ in $\mathscr{B}_+(\hilc)$ by the following iteration:
\begin{equation*}
    \widetilde{B}_0:=\widetilde{B};\quad\widetilde{B}_{n+1}:=\widetilde{B}_n-\widetilde{A}:\widetilde{B}_n=\widetilde{B}_n-(I_C-\widetilde{B}):\widetilde{B}_n,\quad n\in\mathbb{N}.
\end{equation*}
Note that the sequence $\widetilde{B}_n$ is monotone non-increasing, and therefore it strongly converges to a positive operator $\widetilde{B}\in\mathscr{B}_+(\hilc)$.
Using the statement of Lemma \ref{L:2.1} one can easily see that $B_n=J_C^{}\widetilde{B}_nJ_C^*$ for all $n\in\mathbb{N}$, and $B_s=J_C^{}\widetilde{B}_sJ_C^*$. 

In the next result we shall show that $\widetilde{B}_s$ is in fact an orthogonal projection, namely $B_s=P_\mathcal{M}$, where
\begin{equation} \label{M}
\mathcal{M}:=\ker\widetilde{A} \subseteq \hilc.
\end{equation}

Actually, the following statement holds true:
\begin{theorem}\label{T:3.3}
    Let $A,B\in\lefpoz$, $B_n$, $B_s$ and $\mathcal M$ be as above. Then $\widetilde{B}_s=P_\mathcal{M}$, and we have the following factorization of $B_s$:
    \begin{equation}
        B_s=J_CP_\mathcal{M}J_C^*.
    \end{equation}
\end{theorem}
\begin{proof}
    According to Lemma \ref{L:2.1} it is sufficient to prove that $\widetilde{B}_s=P_\mathcal{M}$. Recall from \cite{anderson1975shorted} (or \cite{pekarev1976parallel}) that if $S,T$ are positive operators on a Hilbert space $\hil$ whose sum is invertible, then their parallel sum can be expressed as 
    \begin{equation}\label{E:ParallelInv}
        S:T=S(S+T)^{-1}T=T(S+T)^{-1}S.
    \end{equation}
    This together with equality $\widetilde{A}+\widetilde{B}=I_C$ one gives us
    \begin{equation*}
        \widetilde{A}:\widetilde{B}=\widetilde{A}(\widetilde{A}+\widetilde{B})^{-1}\widetilde{B}=\widetilde{A}\widetilde{B}=\widetilde{B}-\widetilde{B}^2,
    \end{equation*}
    hence
    \begin{equation*}
        \widetilde{B}_1=\widetilde{B}-\widetilde{A}:\widetilde{B}=\widetilde{B}^2.
    \end{equation*}
    One can see that the operator $I_C-\widetilde{B}+\widetilde{B}^2$ is invertible, since
    \begin{equation*}
        I_C-\widetilde{B}+\widetilde{B}^2=\frac{1}{2}(I_C+\widetilde{B}^2+(I_C-\widetilde{B})^2)\ge\frac{1}{2}I_C,
    \end{equation*}
    therefore 
    \begin{equation*}
        \widetilde{A}:\widetilde{B}_1=(I_C-\widetilde{B}):\widetilde{B}^2=(I_C-\widetilde{B})\widetilde{B}^2(I_C-\widetilde{B}+\widetilde{B}^2)^{-1},
    \end{equation*}
    according to \eqref{E:ParallelInv}. Hence
    \begin{equation*}
        \widetilde{B}_2=\widetilde{B}_1-\widetilde{A}:\widetilde{B}_1=\widetilde{B}^2-(I_C-\widetilde{B})\widetilde{B}^2(I_C-\widetilde{B}+\widetilde{B}^2)^{-1}=\widetilde{B}^4(I_C-\widetilde{B}+\widetilde{B}^2)^{-1}.
    \end{equation*}
    In the next step we prove by induction that for all $n\in\mathbb{N}$
    \begin{enumerate}
        \item $I_C-\widetilde{B}+\widetilde{B}_n$ is positive and invertible,
        \item $\widetilde{B}_nB=B\widetilde{B}_n$,
        \item $\widetilde{B}_{n+1}=(I_C-\widetilde{B}+\widetilde{B}_n)^{-1}\widetilde{B}_n^2$.
    \end{enumerate}
    
    (1) We have
    \begin{align*}
        I_C-\widetilde{B}+\widetilde{B}_{n+1}&=I_C-\widetilde{B}+(I_C-\widetilde{B}+\widetilde{B}_n)^{-1}\widetilde{B}_n^2\\
        &=(I_C-\widetilde{B}+\widetilde{B}_n)^{-1}((I_C-\widetilde{B}+\widetilde{B}_n)(I_C-\widetilde{B})+\widetilde{B}_n^2)\\
        &=(I_C-\widetilde{B}+\widetilde{B}_n)^{-1}((I_C-\widetilde{B})^2+\widetilde{B}_n(I_C-\widetilde{B})+\widetilde{B}_n^2)\\
        &=(I_C-\widetilde{B}+\widetilde{B}_n)^{-1}\Bigg(\bigg((I_C-\widetilde{B})+\frac{1}{2}\widetilde{B}_n\bigg)^2+\widetilde{B}_n^2\Bigg).
    \end{align*}
    Using the inequality
    \begin{equation*}
        (I_C-\widetilde{B})+\frac{1}{2}\widetilde{B}_n\ge\frac{1}{2}(I_C-\widetilde{B}+\widetilde{B}_n)
    \end{equation*}
    we can see that $I_C-\widetilde{B}+\widetilde{B}_{n+1}$ is positive and invertible.

    (2) Using formula \eqref{E:ParallelInv} and  by the induction step we have
    \begin{align*}
        \widetilde{B}\widetilde{B}_{n+1}&=\widetilde{B}(I_C-\widetilde{B}+\widetilde{B}_{n+1})^{-1}\widetilde{B}_n^2\\
        &=(I_C-\widetilde{B}+\widetilde{B}_{n+1})^{-1}\widetilde{B}_n^2\widetilde{B}=\widetilde{B}_{n+1}\widetilde{B}.
    \end{align*}
    
    (3) We have
    \begin{equation*}
        \widetilde{A}:\widetilde{B}_{n+1}=(I_C-\widetilde{B}):\widetilde{B}_{n+1}=(I_C-\widetilde{B})\widetilde{B}_{n+1}(I_C-\widetilde{B}+\widetilde{B}_{n+1})^{-1},
    \end{equation*}
    hence
    \begin{align*}
        \widetilde{B}_{n+2}&=\widetilde{B}_{n+1}-\widetilde{A}:\widetilde{B}_{n+1}\\
        &=\widetilde{B}_{n+1}-(I_C-\widetilde{B})\widetilde{B}_{n+1}(I_C-\widetilde{B}+\widetilde{B}_{n+1})^{-1}\\
        &=(I_C-\widetilde{B}+\widetilde{B}_{n+1})^{-1}\widetilde{B}_{n+1}^2,
    \end{align*}
    that proves the recursion formula in (3).
    
    If we take strong limit in $n$ in the recursion formula, we obtain
    \begin{equation*}
        \widetilde{B}_s^{}=(I_C-\widetilde{B}+\widetilde{B}_s^{})^{-1}\widetilde{B}_s^2=(\widetilde{A}+\widetilde{B}_s^{})^{-1}\widetilde{B}_s^2,
    \end{equation*}
    that leads to $\widetilde{A}\widetilde{B}_s=0$, and therefore 
    \begin{equation}\label{E:kerBranA}
        \ker\widetilde{A}^\perp\subseteq \ker\widetilde{B}_s.
    \end{equation}
    On the other hand, we can prove by induction that for any $x\in\ker\widetilde{A}=\ker(I_C-\widetilde{B})$ we have
    \begin{equation*} 
    \widetilde{B}_nx=x.
    \end{equation*}
  Indeed, it is clear for $n=0$. Assume now  $\widetilde{B}_nx=x$ for some $n$, then the recursion formula in (3) gives us
    \begin{align*}
        \widetilde{B}_{n+1}^{}x&=(I_C-\widetilde{B}+\widetilde{B}_n^{})^{-1}\tilde{B}_n^2x=(I_C-\widetilde{B}+\widetilde{B}_n^{})^{-1}x\\
        &=(I_C-\widetilde{B}+\widetilde{B}_n^{})^{-1}(I_C-\widetilde{B}+\widetilde{B}_n^{})x=x.
    \end{align*}
    Passing now to the strong limit gives us 
\begin{equation*} \label{invariance}
        \widetilde{B}_sx=x,\qquad x\in\ker\widetilde{A},
    \end{equation*}
    which, together with \eqref{E:kerBranA}, yields equality $\widetilde{B}_s=P_\mathcal{M}$.
\end{proof}
In Proposition \ref{P:3.1}, we established that the strong limit $B_s$ of the operator sequence $B_n$ obtained through iteration \eqref{E:iteration} is $A$-singular, while $B_a = B - B_s$ is $A$-absolutely continuous. Thus, the equation $B = B_a + B_s$ provides an $A$-Lebesgue decomposition of $B$. In what follows, we will prove that $B_a$ is identical to the maximal absolutely continuous part of $[A]B$. This result simultaneously generalizes Theorem 4.1 of Arlinskii \cite{arlinskii2017} and Theorem 3.1 of Titkos \cite{titkos2019arlinskii}.

\begin{theorem}\label{T:3.4}
    Let $A,B\in\lefpoz$ be positive operators, then we have
    \begin{equation*}
        [A]B=J_C(\widetilde{B}-P_\mathcal{M})J_C^*=B-B_s.
    \end{equation*}
\end{theorem}
\begin{proof}
    According to Lemma \ref{L:2.1} we have
    \begin{equation*}
        [A]B=\lim_{n\to\infty}(nA):B=J_C(\lim_{n\to\infty}(n\widetilde{A}):\widetilde{B})J_C^*=J_C[\widetilde{A}]\widetilde{B}J_C^*.
    \end{equation*}
    Now it is sufficient to prove that
    \begin{equation*}
        \widetilde{B}-P_\mathcal{M}=[\widetilde{A}]\widetilde{B}.
    \end{equation*}
 We will prove this in three steps:
    \begin{enumerate}
        \item $P_\mathcal{M}$ is singular with respect to $I_C-\widetilde{B}$,
        \item $\widetilde{B}-P_\mathcal{M}$ is absolutely continuous with respect to $I_C-\widetilde{B}$,
        \item $[I_C-\widetilde{B}]\widetilde{B}\coloneqq \lim\limits_{n\to\infty}n(I_C-\widetilde{B}):\widetilde{B}\le\widetilde{B}-P_\mathcal{M}$.
    \end{enumerate}

    (1) To see that $P_\mathcal{M}\perp I_C-\widetilde{B}$ it is enough to observe that for any operator $D\geq 0$ satisfying both $D\le P_\mathcal{M}$ and $D\le I_C-\widetilde{B}$ we have both $D|_{\mathcal{M}^\perp}\equiv 0$ and $D|_\mathcal{M}\equiv 0$.
    
    (2) Let $E$ denote the spectral measure of positive contraction  $\widetilde{B}$. We have
    \begin{equation*}
        I_C-\widetilde{B}=\int_0^1 1-t\,dE(t)
    \end{equation*}
    and
    \begin{equation*}
        \widetilde{B}-P_\mathcal{M}=\int_0^1 f(t)\,dE(t)
    \end{equation*}
    where
    \begin{equation} \label{3.1}
        f(t)=\begin{cases} t, & \text{if $t\in[0,1)$,} \\ 0, & \text{if $t=1$.} \\ \end{cases}
    \end{equation}
    Considering functions
    \begin{equation*}
        f_n(t)\coloneqq \max\{t,n(1-t)\},\quad t\in[0,1]
    \end{equation*}
    and operators
    \begin{equation*}
        A_n:=\int_0^1 f_n(t)\,dE(t)
    \end{equation*}
    one can see that sequence $(f_n)_{n\in\mathbb{N}}$ is monotone increasing and pointwise converges to $f$ everywhere on $[0,1]$. Therefore  $(A_n)_{n\in\mathbb{N}}$ is also monotone increasing and converges to $\widetilde{A}=I_C-\widetilde{B}$ in the strong operator topology. On the other hand, $f_n(t)\leq n(1-t)$, hence
    \begin{equation*}
        A_n\le\int_0^1 n(1-t)\,dE(t)=n\widetilde{A},\quad n\in\mathbb{N}.
    \end{equation*}
    Thus we have constructed a monotone increasing sequence of operators $A_n$ strongly converging to $I_C-\widetilde{A}-P_\mathcal{M}=\widetilde{B}-P_\mathcal{M}$, and satisfying $A_n\le n\widetilde{A}$ for all $n\in\mathbb{N}$. Hence $\widetilde{B}-P_\mathcal{M}$ is absolutely continuous with respect to $\widetilde{A}=I_C-\widetilde{B}$.
    
    (3) To see the maximality of $\widetilde{B}-P_\mathcal{M}$ among all absolutely continuous parts it is sufficient to prove that $[I_C-\widetilde{B}]\widetilde{B}\le\widetilde{B}-P_\mathcal{M}$. We have
    \begin{equation*}
        [I_C-\widetilde{B}]\widetilde{B}=\lim_{n\to\infty}n(I_C-\widetilde{B}):\widetilde{B}=\lim_{n\to\infty}\int_0^1\frac{nt(1-t)}{n(1-t)+t}\,dE(t).
    \end{equation*}
    It is easy to check that
    \begin{equation*}
        \frac{nt(1-t)}{n(1-t)+t}\le f(t)
    \end{equation*}
    for every $t\in[0,1]$, where $f$ was defined in \eqref{3.1}. Consequently,
    \begin{equation*}
        [I_C-\widetilde{B}]\widetilde{B}\le\int_0^1 f(t)\,dt=\widetilde{B}P_\mathcal{M},
    \end{equation*}
    which proves (3).

    Finally, we observe that $\widetilde{B}-P_\mathcal{M}\leq [I_C-\widetilde{B}]\widetilde{B}$ due to the maximal property of the latter operator. This observation, together with (3), concludes the proof.
\end{proof}
\begin{corollary} \label{C:3.4}
    Consider $A$, $B$ and $B_s$ as above. We have
    \begin{equation*}
        B-B_s=\max\{C\in\lefpoz\,:\,C\le B,Y\ll A\}.
    \end{equation*}
\end{corollary}

\section{Applications}
In this section, we present some concrete applications of our general results  obtained above.
\subsection{Hilbert spaces}
The first and most direct application of our results concerns the case of operators on Hilbert spaces, as considered by Arlinskii \cite{arlinskii2017}. If $\hil$ is a complex Hilbert space and $A, B$ are positive operators on $\hil$, then, with $E = F = \hil$ and $\dual{\cdot}{\cdot} = \langle \cdot, \cdot \rangle$, $\dual{F}{E}$ forms a $w^*$-sequentially complete anti-dual pair, and $A, B \in \lefpoz$.
Thus, as a direct application of Theorems \ref{T:3.3} and \ref{T:3.4}, we recover Arlinskii's Lebesgue decomposition theorem:
Thus, as a direct application of Theorems \ref{T:3.3} and \ref{T:3.4}, we recover Arlinskii's Lebesgue decomposition theorem:
\begin{theorem}
    Consider $A,B$ be positive operators in the Hilbert space $\hil$ and let $B_s$ be the strong limit of the sequence $\mu_A^{[n]}(B)$ defined as in \eqref{E:iteration} above. Then we have
    \begin{equation*}
        B-B_s=\max\{Y\in\mathcal{B}^+(\mathcal{H})\,:\,Y\le B,Y\ll A\}.
    \end{equation*}
    In particular, the sum $B=(B-B_s)+B_s$ is an $A$-Lebesgue decomposition of $B$.
\end{theorem}

\subsection{Nonnegative forms}
In this subsection, we will show how we can use the results from section 3 to recover \cite{titkos2019arlinskii}*{Theorem 3.1} of Titkos on nonnegative sesquilinear forms. 

Let $X$ be a complex vector space and  denote by $\mathcal{F}^+(X)$ the set of  nonnegative sesquilinear forms on $X$. Consider the following partial ordering of $\mathcal{F}^+(X)$: for two forms $\tform$ and $\wform$ we have $\tform\le\wform$ if $\tform[x]\le\wform[x]$ for all $x\in X$ where $\tform[x]:=\tform(x,x)$ denotes the quadratic form. 

We say the form $\tform$ is $\wform$-closable if
\begin{equation}
    \tform[x_n-x_m]\to 0\quad\text{and}\quad\wform[x_n]\to 0)\quad\mbox{imply}\quad\tform[x_n]\to 0.
\end{equation}
On the contrary,  $\tform$ is $\wform$ are called mutually singular if for all $\sform\in\mathcal{F}^+(X)$ the inequalities $\sform\le\tform$ and $\sform\le\wform$ imply  $\sform=0$ is the zero form. 

Imitating Arlinskii's iterative procedure, Titkos \cite{titkos2019arlinskii} provided a short and elementary proof of the existence of the Lebesgue decomposition for sesquilinear forms. However, this result is somewhat weaker than the Lebesgue decomposition theorem established by Hassi, Sebestyén, and de Snoo \cite{hassi2009lebesgue}*{Theorem 2.11}, as the "maximality property" of the absolutely continuous part was not proven. (Recall that the Lebesgue decomposition is not unique for forms, so this maximality property is not self-evident.)
In what follows, we will show that the decomposition obtained through Titkos's iterative method is, in fact, identical to the extremal Lebesgue decomposition constructed by Hassi, Sebestyén, and de Snoo.


First of all, recall that the parallel sum of two forms $\tform, \wform$ is defined to to be the unique form $\tform:\wform$ whose quadratic form is given by 
\begin{equation*}
    (\tform:\wform)[x]:=\inf\{\wform[x-y]+\tform[y]\,:\,y\in X\},\quad x\in X.
\end{equation*}
Using the above notion we can consider then the mapping $ \mu_{\wform}:\mathcal{F}^+(X)\to\mathcal{F}^+(X)$ defined by 
\begin{equation*}
   \mu_{\wform}(\tform):=\tform-\tform:\wform,
\end{equation*}
and the iteration
\begin{equation*}
    \mu_{\wform}^{[0]}(\tform)=\tform;\qquad\mu_{\wform}^{[n]}(\tform)=\mu_{\wform}(\mu_{\wform}^{[n-1]}(\tform)).
\end{equation*}
Every nonnegative form $\tform\in\mathcal{F}^+(X)$ defines an positive operator $T$ on the anti-dual pair $\dual{\anti X}{X}$, namely by letting
\begin{equation*}
    \langle Tx,y\rangle:=\tform(x,y),\quad x,y\in X.
\end{equation*}
(Here, $\anti X$ stands for the algebraic anti-dual of the vector space $X$.)
Let $\wform$ be another form on $X$, and define the positive operator $W$ in the self-evident way. Then 
\begin{equation*}
    \langle(T:W)x,y\rangle=(\tform:\wform)(x,y),\qquad x,y\in X,
\end{equation*}
and therefore
\begin{equation*}
    \langle\mu_W^{[n]}(T)x,y\rangle=\mu_{\wform}^{[n]}(\tform)(x,y),\qquad x,y\in X,n\in\mathbb{N}.
\end{equation*}
Let $T_s$ denote the corresponding limit of the operator sequence $(\mu_W^{[n]}(T))_{\nen}$, in the sense of section 3, and introduce the form $\tform_s$ by setting
\begin{equation*}
    \tform_s(x,y):=\langle T_s x,y\rangle,\quad x,y\in X.
\end{equation*}
Then  
\begin{equation*}
     \tform_s(x,y)=\limn \mu_{\wform}^{[n]}(\tform)(x,y),\qquad x,y\in X.
\end{equation*}

Applying the results of Section 3 to $T_s$, we  immediately obtain the following statement:
\begin{theorem} \label{T:4.1}
    Let $\tform$ and $\wform$ be forms on $X$. Then $\tform=(\tform-\tform_s)+\tform_s$ is a Lebesgue-type decomposition of $\tform$, where
    \begin{equation*}
        \tform-\tform_s=\max\{\sform\in\mathcal{F}^+(X)\,:\,\sform\le\tform,\sform\ll\wform\}.
    \end{equation*}
\end{theorem}

\subsection{Representable functionals}

Let $\mathcal{A}$ be a (not necessarily unital) *-algebra. We call a linear functional $w:\mathcal{A}\to\mathbb{C}$ positive if $w(a^*a) \geq 0$ for all $a \in \mathcal{A}$. Similarly, for the positive functionals $v$ and $w$, we write $w \leq v$ if $v - w$ is a positive functional. A linear functional $w$ is called representable if there exists a triplet $(\mathcal{H}_w, \pi_w, \zeta_w)$, consisting of a Hilbert space $\mathcal{H}_w$, a *-representation $\pi_w: \mathcal{A} \to \mathcal{B}(\mathcal{H}_w)$, and a vector $\zeta_w \in \mathcal{H}_w$, such that
\begin{equation*}
    w(a)=\sip{\pi_w(a)\zeta_w}{\zeta_w},\qquad a\in\mathcal{A}.
\end{equation*}
We refer to the system $(\mathcal{H}_w, \pi_w, \zeta_w)$ as the GNS triplet corresponding to $w$.

Throughout the following, we will denote by $\mathcal{A}_R^{\dag}$ the partially ordered set of representable functionals. Note that every representable functional is positive, but the converse is not generally true for *-algebras. (In some special cases, such as in a unital Banach *-algebra or a $C^*$-algebra, positivity implies representability. For more details, see \cite{palmer}*{Chapter 9}.) Nevertheless, positive functionals that are majorized by representable functionals are themselves representable. More precisely, if $v \in \mathcal{A}_R^{\dag}$ and $u$ is a positive functional satisfying $u \leq v$, then both $u$ and $v - u$ belong to $\mathcal{A}_R^{\dag}$. Furthermore, the pointwise limit $w$ of a monotone decreasing sequence $(w_n)_{n \in \mathbb{N}}$ of representable functionals is also representable.

If $w, v \in \mathcal{A}_R^{\dag}$ are representable functionals, we say that $w$ is absolutely continuous with respect to $v$ if $v(a_n^*a_n) \to 0$ and $w((a_n - a_m)^*(a_n - a_m)) \to 0$ imply $w(a_n^*a_n) \to 0$. Conversely, we say that $w$ and $v$ are mutually singular if $u \leq w$ and $u \leq v$ imply $u = 0$ for all $u \in \mathcal{A}_R^{\dag}$ (see \cite{gudder1979radon}). For various characterizations of singularity and absolute continuity, as well as a Lebesgue decomposition theorem for representable functionals, the reader is referred to \cite{TARCSAY2020Lebesgue}.
 
Titkos \cite{titkos2019arlinskii} demonstrated that the Lebesgue-type decomposition of representable functionals into absolutely continuous and singular parts can be obtained using Arlinskii's iteration method. However, it should be noted that the Lebesgue-type decomposition of representable functionals (even over $C^*$-algebras) is not unique in general. According to a recent result by Szűcs and Takács \cite{SzucsTakacs}, the Lebesgue-type decomposition is unique for all functionals $w, v \in \mathcal{A}_R^{\dag}$ if and only if every topologically irreducible representation of $\mathcal{A}$ is finite-dimensional. For this reason, in the general case, it is not clear which Lebesgue decomposition is recovered by the iteration method. We show that this procedure also yields the extremal decomposition obtained in \cite{TARCSAY2020Lebesgue}.

Each representable functional $w$ induces a nonnegative form $\tform_w$ on $\mathcal{A}$ in the following natural way:
\begin{equation*}
    \tform_w(a,b):=w(b^*a),\quad a,b\in\mathcal{A}.
\end{equation*}
If $w$ and $v$ are both representable functionals, then there exists a unique representable functional $w:v$, called the parallel sum of $w$ and $v$, that satisfies $\tform_{w:v}=\tform_w:\tform_v$ (see \cite{tarcsay2015parallel}). In other words, $w:v$ satisfies 
\begin{equation*}
    (w:v)(a^*a)=\inf\{w((b^*b)+v(c^*c)\,:\,b,c\in\mathcal{A},b+c=a\}
\end{equation*}
for every $a\in\alg$. Hence the mapping
\begin{equation*}
    \widetilde{\mu}_v:\mathcal{A}_R^{\dag}\to\mathcal{A}_R^{\dag};\quad w\mapsto w-w:v
\end{equation*}
satisfies $\tform_{\widetilde{\mu}_v(w)}=\mu_{\tform_v}(\tform_w)$. Using iteration steps
\begin{equation*}
    \widetilde{\mu}_v^{[0]}(w)\coloneqq w,\quad\widetilde{\mu}_v^{[n+1]}(w):=\widetilde{\mu}_v(\widetilde{\mu}_v^{[n]}(w)),\quad n\in\mathbb{N}
\end{equation*}
we obtain that the pointwise limit
\begin{equation*}
    w_s:=\lim_{n\to\infty}\widetilde{\mu}_v^{[n]}(w)
\end{equation*}
is a representable functional, moreover,
\begin{equation*}
    \tform_{w_s}=\lim_{n\to\infty}\mu_{\tform_v}^{[n]}(\tform_w).
\end{equation*}

This enables us to invoke Theorem \ref{T:4.1} to obtain the following result generalizing \cite{titkos2019arlinskii}*{Theorem 4.1} (cf. also \cite{tarcsay2016lebesgue}*{Theorem 3.3}):
\begin{theorem}
    Let $v$ and $w$ be representable functionals of $\mathcal{A}$. Then $$w=w_s+(w-w_s)$$ is a Lebesgue-type decomposition of $w$, moreover
    \begin{equation*}
        w-w_s=\max\{u\in\mathcal{A}^{\dag}_R\,:\,u\le w,u\ll v\}.
    \end{equation*}
\end{theorem}

\end{document}